\documentclass[11pt]{amsart}

\usepackage[margin=1.12in]{geometry}
\usepackage{amsmath,amssymb,amsthm,mathtools,mathrsfs}
\usepackage{enumitem}
\usepackage{booktabs}
\usepackage{tikz-cd}
\usepackage{xcolor}
\usepackage{hyperref}
\usepackage{orcidlink}

\definecolor{darkgoldenrod}{rgb}{0.72,0.53,0.04}
\definecolor{goldmetallic}{rgb}{0.83,0.69,0.22}
\hypersetup{colorlinks=true,linkcolor=darkgoldenrod,filecolor=brown,urlcolor=goldmetallic,citecolor=darkgoldenrod}
\numberwithin{equation}{section}

\newcommand{\Q}{\mathbb Q}
\newcommand{\Z}{\mathbb Z}
\newcommand{\Fp}{\mathbb F_p}
\newcommand{\Qp}{\mathbb Q_p}
\newcommand{\Zp}{\mathbb Z_p}
\newcommand{\Gal}{\operatorname{Gal}}
\newcommand{\Hom}{\operatorname{Hom}}
\newcommand{\Ad}{\operatorname{Ad}}

\newcommand{\cE}{\mathcal E}
\newcommand{\frakg}{\mathfrak g}
\newcommand{\frakt}{\mathfrak t}
\newcommand{\chibar}{\overline\chi}

\newcommand{\Ann}{\operatorname{Ann}}
\newcommand{\Ass}{\operatorname{Ass}}

\newcommand{\GG}{\mathbf G}
\newcommand{\TT}{\mathbf T}
\newcommand{\Gder}{\mathbf G^{\mathrm{der}}}

\newcommand{\Gm}{\mathbf G_m}
\newcommand{\Lie}{\operatorname{Lie}}
\newcommand{\Ht}{\operatorname{ht}}
\newcommand{\SL}{\operatorname{SL}}
\newcommand{\GL}{\operatorname{GL}}

\newtheorem{theorem}{Theorem}[section]
\newtheorem{proposition}[theorem]{Proposition}
\newtheorem{lemma}[theorem]{Lemma}
\newtheorem{corollary}[theorem]{Corollary}
\theoremstyle{definition}

\newtheorem{remark}[theorem]{Remark}

\title[Galois representations ramified at one prime]{Galois representations ramified at one prime via relative deformation theory}
\author{Anwesh Ray}
\address{Chennai Mathematical Institute, H1, SIPCOT IT Park, Siruseri, Kelambakkam 603103, Tamil Nadu, India}
\email{anwesh@cmi.ac.in}
\date{}

\begin{document}

\begin{abstract}
Let $p$ be an odd prime and let $\GG/\Z$ be a split connected reductive
group with $\dim Z(\GG)\leq1$. Assume that a split maximal torus of
$\GG$ admits a cocharacter whose pairing with every simple root is
odd, and impose an explicit root-theoretic condition on the reduction
of $\Lie(\Gder)$ modulo $p$. We construct infinitely many continuous
representations $\rho:G_{\{p\}}\longrightarrow\GG(\Zp)$
which are unramified at every finite prime different from $p$ and whose
images contain a principal congruence subgroup of $\Gder(\Zp)$. When
the center has dimension one, the representations may be chosen to have
open image in $\GG(\Zp)$.

The construction continues the author's earlier work on
$\GL_n$-valued representations ramified at one prime, but replaces the
residual unobstructedness used there by a relative lifting argument. For $\GG=\GL_n$ all the required root-theoretic conditions are
automatic for every odd prime. We therefore obtain, for every odd
$p$ and every $n>1$, infinitely many representations
$G_{\{p\}}\to\GL_n(\Zp)$ with open image. In particular, this removes
the weak Vandiver-type hypothesis occurring in
\cite{Ray2023}, as well as the restriction $p\geq7$ in that
construction.
\end{abstract}

\maketitle

\section{Introduction}\label{sec:introduction}

\par Let $p$ be a prime and let $\GG$ be a reductive group. The problem of
constructing Galois representations with large image in
$\GG(\Zp)$ may be viewed as a $p$-adic analogue of the inverse Galois
problem. In the classical inverse Galois problem one asks which finite
groups occur as Galois groups over $\Q$. Here the target is instead a
compact $p$-adic group, and one asks for continuous representations of
the absolute Galois group whose image is open, or at least contains a
principal congruence subgroup of the derived group. A natural
refinement is to prescribe, or minimize, the set of primes at which the
representation is allowed to ramify.

\par There are many natural sources of Galois representations in arithmetic
geometry. When the dimension is two, representations with large image
arise abundantly from elliptic curves and modular forms. In higher
dimensions one may consider Galois representations associated with
motives, automorphic forms, or the Tate modules of abelian varieties.
These geometric constructions, however, frequently impose additional
structure on the image. For example, the representation on the Tate
module of a polarized abelian variety is constrained by the Weil
pairing. It is therefore difficult, by purely geometric means, to
produce higher-dimensional representations into $\GL_n(\Zp)$ whose
images contain open subgroups of $\SL_n(\Zp)$ and which at the same
time satisfy very restrictive ramification conditions. This suggests
a different point of view. Rather than insisting that the
representation arise from a geometric or automorphic object, one can
construct it directly by Galois-theoretic and deformation-theoretic
methods.

\par The ramification problem considered in this paper is obtained
by allowing ramification at only one finite prime. Let
$\Q_{\{p\}}$ be the maximal algebraic extension of $\Q$ which is
unramified at every finite prime different from $p$, and put $G_{\{p\}}
        :=
        \Gal(\Q_{\{p\}}/\Q)$. Thus a representation of $G_{\{p\}}$ is a representation of
$G_\Q$ which is unramified at every finite prime $\ell\neq p$. The
question is whether the very severe restriction on ramification is
compatible with having a large $p$-adic image.

\par This problem originates in work of Greenberg
\cite{Greenberg2016}. Let
$\bar\chi$ denote the mod-$p$ cyclotomic character and let $C$ be the
mod-$p$ class group of $\Q(\mu_p)$. The action of
$\Gal(\Q(\mu_p)/\Q)$ decomposes $C$ into eigenspaces
$C(\bar\chi^i)$. The prime $p$ is regular precisely when $C=0$.
Vandiver's conjecture predicts the vanishing of the even eigenspaces,
or equivalently of the $p$-primary part of the class group of the
maximal totally real subfield of $\Q(\mu_p)$. Thus the arithmetic of
the cyclotomic class group enters naturally into the attempt to
construct one-prime Galois representations. A theorem of Shafarevich
implies that when $p$ is regular, the maximal pro-$p$ extension of
$\Q(\mu_p)$ unramified outside the primes above $p$ has free pro-$p$ Galois
group of explicitly known rank. Greenberg exploits this freeness to
construct representations $G_{\{p\}}\longrightarrow\GL_n(\Zp)$
whose images contain an open subgroup of $\SL_n(\Zp)$, provided that
$p$ is regular and 
$p\geq4\lfloor n/2\rfloor+1$.

\par Motivated by Greenberg's construction, the author formulated in
\cite{Ray2021} the following basic problem: for a given prime $p$ and
$n>1$, does there exist a continuous representation
$G_{\{p\}}\to\GL_n(\Zp)$ whose image contains an open subgroup of
$\SL_n(\Zp)$? More generally, how small can the ramification set of a
$p$-adic representation with large image be?

\par The first deformation-theoretic treatment of this question was given
in \emph{loc.cit}. Instead of requiring the entire mod-$p$ class group
of $\Q(\mu_p)$ to vanish, one begins with a carefully chosen diagonal
residual representation and imposes vanishing only on the particular
cyclotomic eigenspaces which occur in its adjoint representation. If
$e_p$ denotes the index of irregularity, this gives infinitely many
representations unramified outside $p$ provided that $p$ is
sufficiently large in terms of $n$ and $e_p$. The essential
deformation-theoretic input is the vanishing $H^2(G_{\{p\}},\Ad\bar\rho)=0$ for the chosen residual representation $\bar\rho$. Once this
unobstructedness has been established, the remaining part of the
argument consists in modifying level lifts
by suitable $H^1$ Galois cohomology classes and then proving that the resulting image
contains a principal congruence subgroup.

\par The dependence on the irregularity index and on the size of $p$ was
substantially reduced in the author's subsequent paper
\cite{Ray2023}. There it is shown that for every $n>1$ and every
$p\geq7$ there are infinitely many representations
$G_{\{p\}}\to\GL_n(\Zp)$ whose images contain an explicitly bounded
principal congruence subgroup of $\SL_n(\Zp)$, subject, when
$p\equiv1\pmod4$, to a weak form of Vandiver's conjecture.
\par There are several related developments which help place this question
in context. Katz \cite{Katz2019} constructs finitely ramified
representations with large image by geometric methods and Hilbert
irreducibility, although the resulting ramification set is not
restricted to the single prime $p$. Maire \cite{Maire2023} gives a
different unconditional construction of representations with large
image in $\GL_n(\Zp)$. For $p\equiv3\pmod4$ the representations may be
taken to be unramified outside $p$, while for
$p\equiv1\pmod4$ his construction allows the additional finite prime
$2$. Maletto \cite{Maletto2023} develops a particularly explicit
version of the residual method for $\mathrm{GSp}_4$.

\par Of particular relevance here is Tang's extension of the
deformation-theoretic construction to reductive groups
\cite{Tang2023}. Let $\GG$ be split reductive with
$\dim Z(\GG)\leq1$. Tang starts with a cyclotomic representation in a
split maximal torus, decomposes the adjoint representation into its
root spaces, introduces root directions by Galois cohomology classes,
and uses the Chevalley commutator relations to obtain open image. In
this sense, Tang identifies the natural root-theoretic form of the
construction in \cite{Ray2021}. In this construction, the residual representation is chosen so that
the corresponding $H^2$ vanishes. This requires avoiding the
cyclotomic class-group eigenspaces that may contribute to the
obstruction and leads to a lower bound on $p$ depending on the root
datum and a bound for the irregularity index.

\par The aim of the present paper is to remove any regularity assumptions
on $p$ and prove results that do not assume any form of Vandiver's
conjecture. The deformation-theoretic approach belongs to the
Galois-cohomological lifting tradition initiated by Ramakrishna
\cite{Ramakrishna1999}. The argument used here is particularly reminiscent
of the relative deformation philosophy of
Fakhruddin--Khare--Patrikis \cite{FKP2021}. The key advantage is that
one no longer requires the obstruction group itself to vanish. As one
passes between representations modulo successive powers of $p$, there
are natural comparison maps between the corresponding adjoint modules
and hence between their second cohomology groups. At sufficiently high
levels these cohomological comparison maps become injective. Thus an
obstruction which vanishes after passing to the lower level must
already vanish at the higher level, and this is enough to continue the
lifting process even when the deformation problem is genuinely
obstructed.

\par We now state the main theorem. We first recall explicitly the two
root-theoretic conditions which enter its formulation. Let
$\TT\subset\GG$ be a split maximal torus, let
$\Phi=\Phi(\GG,\TT)$ be the associated root system, and let
$\Delta\subset\Phi$ be a set of simple roots. We say that
$(\GG,\TT,\Delta)$ \emph{admits an odd cocharacter} if there exists
$\lambda_0\in X_*(\TT)$ such that $\langle\alpha,\lambda_0\rangle
        \equiv1\pmod 2$ for all $
        \alpha\in\Delta$. The odd-cocharacter condition is automatic for several standard
classical groups. For the usual choices of maximal torus and simple
roots, one can in fact choose a cocharacter whose pairing with every
simple root is equal to $1$.  This holds for $\GL_n$,
$\mathrm{SO}_{2n+1}$, $\mathrm{SO}_{2n}$, and $\mathrm{GSp}_{2n}$.
\par Fix a pinning and corresponding Chevalley root vectors
$X_\alpha$, $\alpha\in\Phi$. Whenever
$\alpha,\beta,\alpha+\beta\in\Phi$, write
\[
        [X_\alpha,X_\beta]
        =
        N_{\alpha,\beta}X_{\alpha+\beta},
\]
and for $\delta\in\Delta$ put
$H_\delta=[X_\delta,X_{-\delta}]$. We say that an odd prime $p$ is
\emph{Lie-good for $\GG$} if every nonzero structure constant
$N_{\alpha,\beta}$ occurring in these root-generation relations is
nonzero modulo $p$, and if the reductions of the root vectors
$X_\alpha$ together with the simple coroot directions $H_\delta$ span $\Lie(\Gder)\otimes_{\Z}\Fp$.
For a fixed split integral
root datum it excludes only finitely many primes, and for
$\GG=\GL_n$ it is automatic for every prime.

\begin{theorem}\label{thm:main}
Let $\GG/\Z$ be a split connected reductive group with
$\dim Z(\GG)\leq1$. Fix a split maximal torus $\TT\subset\GG$ and a
set of simple roots $\Delta$. Assume that $(\GG,\TT,\Delta)$ admits an
odd cocharacter. Let $p$ be an odd prime which is Lie-good for $\GG$.
Then there exist infinitely many continuous representations $\rho:G_{\{p\}}\longrightarrow\GG(\Zp)$
whose images contain a principal congruence subgroup of
$\Gder(\Zp)$. If $\dim Z(\GG)=1$, the representations may be chosen so
that their images in $(\GG/\Gder)(\Zp)$ are open; in particular their
images are open in $\GG(\Zp)$.
\end{theorem}

We obtain the following
consequence for $\mathrm{GL}_n$.

\begin{corollary}\label{cor:GLn}
Let $p$ be any odd prime and let $n>1$. Then there exist infinitely
many continuous representations $\rho:G_{\{p\}}\longrightarrow\GL_n(\Zp)$
unramified outside $\{p,\infty\}$ whose images contain a principal
congruence subgroup of $\SL_n(\Zp)$. The determinant may be chosen to
have open image, and hence the resulting representations have open
image in $\GL_n(\Zp)$.
\end{corollary}

\section{Preliminaries}\label{sec:preliminaries}

In this section we discuss preliminary notions as well as set up notation used throughout the article.

Let $\overline\Q$ be a fixed algebraic closure of $\Q$ and let
$G_\Q=\Gal(\overline\Q/\Q)$. For a finite set $S$ of rational primes,
let $\Q_S$ be the maximal algebraic extension of $\Q$ contained in
$\overline\Q$ which is unramified at every finite prime outside $S$,
and set $G_S=\Gal(\Q_S/\Q)$. A continuous representation of
$G_\Q$ which is unramified at every finite prime
$\ell\neq p$ factors uniquely through $G_{\{p\}}$.

For every rational prime $\ell$ choose an embedding
$\overline\Q\hookrightarrow\overline\Q_\ell$ and thereby a
decomposition subgroup $G_\ell\subset G_\Q$.  We write $I_\ell$ for
its inertia subgroup.  The $p$-adic cyclotomic character is
\[
        \chi:G_{\{p\}}\longrightarrow\Zp^\times,
\]
and $\chibar$ denotes its reduction modulo $p$. Set $F_\infty=\Q(\mu_{p^\infty})$ and $
 \Gamma=\Gal(F_\infty/\Q)$.
Since $p$ is odd, we have that $\Gamma=\Delta_p\times\Gamma_1$, where $\Delta_p\simeq(\Z/p\Z)^\times$ and $\Gamma_1\simeq1+p\Zp\simeq\Zp$.
Let $\omega$ be the Teichmuller character of $\Delta_p$. Choose a
topological generator $\gamma$ of $\Gamma_1$ and set
$u=\chi(\gamma)$. We write
$\Lambda=\Zp[[\Gamma_1]]\simeq\Zp[[T]]$, with
$\gamma^{-1}=1+T$. Finally set $H=\Gal(\Q_{\{p\}}/F_\infty)$. 

Let $\GG/\Z$ be a split connected reductive group and let
$\TT\subset\GG$ be a split maximal torus. Fix a Borel subgroup
containing $\TT$ and let $\Phi=\Phi(\GG,\TT)$, $\Delta\subset\Phi$ and $\Phi^+\subset\Phi$ be the associated root system, simple roots, and positive roots respectively. Let
$\Gder$ be the derived subgroup. We set $\frakg_{\Zp}=\Lie(\GG)\otimes\Zp$, $\frakt_{\Zp}=\Lie(\TT)\otimes\Zp$ and $\frakg^{\mathrm{der}}_{\Zp}=\Lie(\Gder)\otimes\Zp$. For a coefficient ring $R$ we use the corresponding subscript $R$.

Choose a pinning. For every root $\alpha\in\Phi$ this gives a root
vector $X_\alpha$ spanning a free rank-one root submodule.  The
integral root decomposition is
\begin{equation}\label{eq:root-decomposition-integral}
        \frakg_{\Zp}
        =
        \frakt_{\Zp}
        \oplus
        \bigoplus_{\alpha\in\Phi}\Zp X_\alpha.
\end{equation}
If $\alpha+\beta\in\Phi$, the Chevalley relation has the form
\begin{equation}\label{eq:Chevalley-root-bracket}
        [X_\alpha,X_\beta]
        =N_{\alpha,\beta}X_{\alpha+\beta},
\end{equation}
where $N_{\alpha,\beta}\in\Z\setminus\{0\}$, and
\begin{equation}\label{eq:opposite-root-bracket}
        [X_\alpha,X_{-\alpha}]=H_\alpha
        \in\frakt_{\Zp}\cap\frakg^{\mathrm{der}}_{\Zp}.
\end{equation}
For $\alpha=\sum_{\delta\in\Delta}n_\delta\delta\in\Phi^+$ define
\[
        \Ht(\alpha)=\sum_{\delta\in\Delta}n_\delta,
\]
and set $\Ht(-\alpha)=-\Ht(\alpha)$. Let
\[
        X_*(\TT)
        :=
        \operatorname{Hom}_{\mathrm{alg.grp.}}(\mathbf G_m,\TT)
\]
denote the cocharacter lattice of the split maximal torus $\TT$, and
let
\[
        X^*(\TT)
        :=
        \operatorname{Hom}_{\mathrm{alg.grp.}}(\TT,\mathbf G_m)
\]
denote its character lattice. There is a natural perfect pairing
\[
        \langle\ ,\ \rangle:
        X^*(\TT)\times X_*(\TT)\longrightarrow\Z
\]
characterized by the identity
\[
        \alpha\circ\lambda(t)
        =
        t^{\langle\alpha,\lambda\rangle}
\]
for $\alpha\in X^*(\TT)$ and $\lambda\in X_*(\TT)$. In particular,
each root $\alpha\in\Phi(\GG,\TT)\subset X^*(\TT)$ may be paired with
a cocharacter.

\par We write
\[
        \Z^\Delta
        :=
        \prod_{\alpha\in\Delta}\Z,
\]
whose coordinates are indexed by the simple roots, and consider the
homomorphism
\[
        \vartheta:X_*(\TT)\longrightarrow\Z^\Delta
\]
defined by $\lambda\longmapsto
        \bigl(\langle\alpha,\lambda\rangle\bigr)_{\alpha\in\Delta}$. Thus an odd cocharacter is precisely a cocharacter $\lambda$ for which
every coordinate of $\vartheta(\lambda)$ is odd. Equivalently, after
reducing the above map modulo $2$, its image is the vector
\[
        (1,\ldots,1)\in(\Z/2\Z)^\Delta.
\]
Let $\lambda\in X_*(\TT)$ be a cocharacter. Composing with the
cyclotomic character gives
\begin{equation}\label{eq:toral-general}
        \rho_\lambda=\lambda\circ\chi:
        G_{\{p\}}\longrightarrow\TT(\Zp)\subset\GG(\Zp).
\end{equation}
For a root $\alpha$ put
\[
        d_\alpha=\langle\alpha,\lambda\rangle\in\Z.
\]
Conjugation by $\lambda(t)$ on the root line is multiplication by
$t^{d_\alpha}$. Consequently
\begin{equation}\label{eq:Ad-toral-decomposition}
 \Ad\rho_\lambda(\frakg_{\Zp})
 =
 \frakt_{\Zp}
 \oplus
 \bigoplus_{\alpha\in\Phi}\Zp(d_\alpha)X_\alpha,
\end{equation}
where $G_{\{p\}}$ acts trivially on $\frakt_{\Zp}$ and through
$\chi^{d_\alpha}$ on the root line.  This decomposition is the basic
bridge between the root datum and the cyclotomic cohomology considered
in Section~\ref{sec:cyclotomic}.

When $\dim Z(\GG)=1$, put
$\mathbf S=\GG/\Gder$. Since $\GG$ is split, $\mathbf S$ is a
one-dimensional split torus. The image of $\lambda$ in
$X_*(\mathbf S)$ will be denoted $\lambda^{\mathrm{ab}}$. When $\lambda^{\mathrm{ab}}\neq0$, the composition
\[
 G_{\{p\}}\xrightarrow{\rho_\lambda}\GG(\Zp)
 \longrightarrow\mathbf S(\Zp)
\]
is a nonzero integral power of $\chi$ and hence has open image.

\par We now make precise the second root-theoretic condition appearing in
the main theorem. Recall that $\Phi=\Phi(\GG,\TT)$ is the root system of
$\GG$ with respect to the fixed split maximal torus $\TT$, and that
$\Delta\subset\Phi$ is the chosen set of simple roots. Put $\frakg_{\Z}:=\Lie(\GG)$ and $\frakg_{\Z}^{\mathrm{der}}
        :=\Lie(\Gder)$, where $\Gder$ denotes the derived subgroup of $\GG$. Thus
$\frakg_{\Z}^{\mathrm{der}}$ is the integral Lie algebra of the
semisimple group $\Gder$.

\par Fix the pinning chosen above. For every root
$\alpha\in\Phi$, let $X_\alpha\in\frakg_{\Z}^{\mathrm{der}}$ denote the corresponding Chevalley root vector. The root submodule
associated with $\alpha$ is the free rank-one $\Z$-module
$\Z X_\alpha$. For each simple root $\delta\in\Delta$, put
\[
        H_\delta
        :=
        [X_\delta,X_{-\delta}]
        \in
        \Lie(\TT)\cap\frakg_{\Z}^{\mathrm{der}}.
\]
Thus the vectors $X_\alpha$ give the root directions in the derived
Lie algebra, while the elements $H_\delta$ give the simple coroot
directions.

\par Define
\[
        \mathcal L_{\mathrm{root}}
        :=
        \left\langle
        X_\alpha\;(\alpha\in\Phi),
        \ H_\delta\;(\delta\in\Delta)
        \right\rangle_{\Z}
        \subseteq
        \frakg_{\Z}^{\mathrm{der}}.
\]
After tensoring with $\Q$, the root vectors together with the simple
coroot directions span the full semisimple Lie algebra. Hence
\[
        \mathcal L_{\mathrm{root}}\otimes_{\Z}\Q
        =
        \frakg_{\Z}^{\mathrm{der}}\otimes_{\Z}\Q.
\]
It follows that $\mathcal L_{\mathrm{root}}$ has finite index in
$\frakg_{\Z}^{\mathrm{der}}$. We denote this index by
\[
        c_{\GG}
        :=
        \left[
        \frakg_{\Z}^{\mathrm{der}}:
        \mathcal L_{\mathrm{root}}
        \right].
\]
If $p\nmid c_{\GG}$, reduction modulo $p$ shows that the images of the
root vectors $X_\alpha$ and the simple coroot directions $H_\delta$
span $\frakg_{\Fp}^{\mathrm{der}}
        :=
        \frakg_{\Z}^{\mathrm{der}}\otimes_{\Z}\Fp$. This is one part of the Lie-good condition.

\par The other part concerns the Chevalley structure constants. Whenever
$\alpha,\beta,\alpha+\beta\in\Phi$, one has
\[
        [X_\alpha,X_\beta]
        =
        N_{\alpha,\beta}X_{\alpha+\beta},
\]
where $N_{\alpha,\beta}\in\Z\setminus\{0\}$. Set
\[
        N_{\GG}
        :=
        \prod_{\substack{\alpha,\beta\in\Phi\\
                          \alpha+\beta\in\Phi}}
        |N_{\alpha,\beta}|.
\]
Thus $N_{\GG}$ is a positive integer depending only on the chosen
integral root datum. If $p\nmid c_{\GG}N_{\GG}$, then every Chevalley constant which occurs in the root-generation
argument is nonzero modulo $p$, and the root and coroot directions
span $\frakg_{\Fp}^{\mathrm{der}}$. Consequently every odd prime
$p\nmid c_{\GG}N_{\GG}$ is Lie-good for $\GG$.

\par For $\GG=\mathrm{GL}_n$, we have that $X_{e_i-e_j}=e_{ij}$. Whenever a root sum occurs the corresponding structure constant is
$\pm1$; for example, $[e_{ij},e_{jk}]=e_{ik}$. Moreover
\[
        \{e_{ij}:i\neq j\}
        \cup
        \{e_{ii}-e_{i+1,i+1}:1\leq i<n\}
\]
is an integral basis of $\mathfrak{sl}_n$. Therefore $c_{\GL_n}=1$ and $N_{\GL_n}=1$. The Lie-good condition is consequently automatic for every prime in
the general linear case. 

For $r\ge1$ define
\[
 U_r(\GG)=\ker\!\left(\GG(\Zp)\to\GG(\Z/p^r\Z)\right)
\]
and similarly $U_r(\Gder)$.  Smoothness of the split reductive group
scheme gives a canonical identification of additive groups
\begin{equation}\label{eq:congruence-quotient}
 U_r(\GG)/U_{r+1}(\GG)
 \simeq\frakg_{\Fp}.
\end{equation}
We recall this identification in a form which will also be used for
square-zero coefficient rings.  Let $I=p^r\Z/p^{r+1}\Z$.  Since
$I^2=0$, the kernel of
$\GG(\Z/p^{r+1}\Z)\to\GG(\Z/p^r\Z)$ is naturally the tangent space
$\frakg_{\Fp}\otimes I$.  Thus an element $x\in U_r(\GG)$ has a
well-defined leading term
\[
        \ell_r(x)\in\frakg_{\Fp}.
\]
Multiplication in the congruence quotient corresponds to addition of
leading terms.

Two elementary calculations will be used repeatedly.  If
$x\in U_r(\GG)$ and $y\in U_s(\GG)$ have leading terms $X$ and $Y$,
then
\begin{equation}\label{eq:commutator-leading-general}
        [x,y]\in U_{r+s}(\GG),
        \quad\text{and}\quad
        \ell_{r+s}([x,y])=[X,Y].
\end{equation}
This may be checked after fixing any faithful integral representation
$\GG\hookrightarrow\GL_N$; the matrix computation is independent of
the chosen embedding because it is the intrinsic Lie bracket on the
tangent space.  Likewise, for odd $p$,
\begin{equation}\label{eq:p-power-leading-general}
        x^p\in U_{r+1}(\GG),
        \quad\text{and}\quad
        \ell_{r+1}(x^p)=X.
\end{equation}
Indeed, after embedding in $\GL_N$, write
$x=1+p^r\widetilde X+O(p^{r+1})$.  The binomial expansion gives
\[x^p=1+p^{r+1}\widetilde X+O(p^{r+2}).\] For
$2\le k\le p-1$, the term
$\binom pk(p^r\widetilde X)^k$ has valuation at least
$1+kr\ge r+2$, and the $k=p$ term has valuation
$pr\ge r+2$ because $p$ is odd.  The calculation also shows why $p=2$
is excluded from the uniform statement.

If $D\subset\Gder(\Zp)$ is a closed subgroup, define
\begin{equation}\label{eq:Phi-general}
 \Phi_r(D)
 =
 \{X\in\frakg^{\mathrm{der}}_{\Fp}:
 \text{there exists }d\in D\cap U_r(\Gder)
 \text{ with }\ell_r(d)=X\}.
\end{equation}
Multiplication and taking integral powers show that
$\Phi_r(D)$ is an $\Fp$-subspace. Equations
\eqref{eq:commutator-leading-general} and
\eqref{eq:p-power-leading-general} imply
\begin{equation}\label{eq:Phi-basic-properties}
 [\Phi_r(D),\Phi_s(D)]\subset\Phi_{r+s}(D),
  \quad\text{and}\quad
 \Phi_r(D)\subset\Phi_{r+1}(D).
\end{equation}

\section{Finite-level deformation theory and relative obstruction classes}
\label{sec:finite-deformation}

The relative lifting argument is the main point at which the present paper
differs from the residual constructions in \cite{Ray2021,Ray2023,Tang2023}.
Throughout this section $\mathcal G$ denotes an arbitrary profinite
group and $\GG/\Zp$ is a smooth affine group scheme; eventually
$\mathcal G=G_{\{p\}}$ and $\GG$ is our split reductive group.

Let $B'\twoheadrightarrow B$ be a surjection of finite local
$\Zp$-algebras with kernel $I$, and assume $I^2=0$.  Smoothness gives
surjectivity $\GG(B')\twoheadrightarrow \GG(B)$.
The kernel is canonically the additive group $\frakg\otimes I$.
We denote by
\[
        \kappa_I:
        \frakg\otimes I
        \longrightarrow
        \ker\!\left(\GG(B')\to\GG(B)\right)
\]
the inverse of this canonical identification. Thus
\[
 \kappa_I(X)\kappa_I(Y)=\kappa_I(X+Y)
\]
for $X,Y\in\frakg\otimes I$. If $g\in\GG(B)$ and $\widetilde g\in\GG(B')$ is any lift, then
conjugation by $\widetilde g$ on the kernel agrees with the adjoint
action of $g$ on $\frakg\otimes I$.  The action is independent of the
choice of lift.

Let $\tau:\mathcal G\longrightarrow\GG(B)$ be a continuous representation. We find that
$\frakg\otimes I$ is a continuous $\mathcal G$-module, which we denote
$\Ad\tau(\frakg\otimes I)$. Since the image of $\tau$ is finite, we
may choose a set-theoretic section of $\GG(B')\to\GG(B)$ on that image
and obtain a continuous map
\[
        \widetilde\tau:\mathcal G\longrightarrow\GG(B')
\]
lifting $\tau$ and satisfying $\widetilde\tau(1)=1$.  It need not be a
homomorphism.

\begin{lemma}\label{lem:obstruction}
There is a canonically defined class
\[
 o(\tau,B'/B)\in
 H^2\!\left(\mathcal G,\Ad\tau(\frakg\otimes I)\right)
\]
which vanishes if and only if $\tau$ lifts to a continuous
representation $\mathcal G\to\GG(B')$.
\end{lemma}

\begin{proof}
For $g,h\in\mathcal G$, both
$\widetilde\tau(g)\widetilde\tau(h)$ and
$\widetilde\tau(gh)$ reduce to $\tau(gh)$ in $\GG(B)$.  Hence there is
a unique $C(g,h)\in\frakg\otimes I$ such that
\begin{equation}\label{eq:def-obstruction-cocycle}
 \widetilde\tau(g)\widetilde\tau(h)
 =\kappa_I(C(g,h))\widetilde\tau(gh).
\end{equation}
Associativity of multiplication in $\GG(B')$, together with the
additivity of the square-zero kernel, gives
\[
 C(g,h)+C(gh,k)
 =g\cdot C(h,k)+C(g,hk).
\]
Thus $C$ is a continuous $2$-cocycle.

If $\widetilde\tau'$ is another lift, there is a unique continuous
$1$-cochain $b:\mathcal G\to\frakg\otimes I$ with
\[
 \widetilde\tau'(g)=\kappa_I(b(g))\widetilde\tau(g).
\]
Substitution into \eqref{eq:def-obstruction-cocycle} shows that the
new obstruction cocycle is $C+db$.  Therefore the class $[C]$ is
independent of all choices; this is $o(\tau,B'/B)$.

If the class vanishes, write $C=db$. Replacing
$\widetilde\tau(g)$ by $\kappa_I(-b(g))\widetilde\tau(g)$ kills the
obstruction and produces a genuine homomorphism. Conversely a genuine lift
has zero obstruction.  This proves the lemma.
\end{proof}

\begin{remark}\label{rem:not-schlessinger-small}
The calculation does not require $B'\to B$ to be a small extension
in the Schlessinger sense. The only property used is $I^2=0$.
\end{remark}

\begin{lemma}\label{lem:obstruction-naturality}
Suppose
\[
        B'\twoheadrightarrow B''\twoheadrightarrow B
\]
are surjections of finite local rings. Put
$I=\ker(B'\to B)$ and $I''=\ker(B''\to B)$ and assume $I^2=0$.
Then the natural map $I\to I''$ carries
$o(\tau,B'/B)$ to $o(\tau,B''/B)$.
\end{lemma}

\begin{proof}
Choose one set-theoretic lift of $\tau$ to $\GG(B')$ and reduce it to
$\GG(B'')$.  The multiplicative obstruction of the reduced lift is exactly
the image of the original obstruction under
$\frakg\otimes I\to\frakg\otimes I''$.  Passing to cohomology gives
the assertion.
\end{proof}

Fix now an integral representation
\[
        \rho_\infty:\mathcal G\longrightarrow\GG(\Zp)
\]
and set
\[
        T=\Ad\rho_\infty(\frakg_{\Zp}),
        \qquad T_m=T/p^mT,
        \quad \text{and}\quad \rho_m=\rho_\infty\pmod{p^m}.
\]
The following elementary calculation is the coefficient bookkeeping
behind the relative method.

\begin{lemma}\label{lem:window-coeff}
Let $R\ge2$ and $j\ge R$. Set
\[
 B'=\Z/p^{j+R+1}\Z,
 \qquad
 B=\Z/p^{j+1}\Z,
 \quad \text{and}\quad 
 I=p^{j+1}\Z/p^{j+R+1}\Z.
\]
Then $I^2=0$, division by $p^{j+1}$ identifies $I$ additively with
$\Z/p^R\Z$, and if
$\tau:\mathcal G\to\GG(B)$ satisfies
$\tau\bmod p^R=\rho_R$, then
\[
        \Ad\tau(\frakg\otimes I)\simeq T_R
\]
as $\mathcal G$-modules.
\end{lemma}

\begin{proof}
Because $j\ge R$,
\[
        2(j+1)\ge j+R+1,
\]
so the product of two elements of $I$ is divisible by
$p^{j+R+1}$ and hence $I^2=0$.  Every element of the tangent kernel may
be written uniquely as $p^{j+1}X$ with $X$ defined modulo $p^R$.
Conjugating by a lift of $\tau(g)$ gives
\[
 p^{j+1}X\longmapsto
 p^{j+1}\Ad(\widetilde{\tau(g)})(X).
\]
After division by $p^{j+1}$ and reduction modulo $p^R$, only
$\tau(g)\bmod p^R=\rho_R(g)$ remains.  Thus the action is the adjoint
action through $\rho_R$, which is exactly $T_R$.
\end{proof}

We can now state the fixed-width lifting proposition. Its point is
that the obstruction group at width $R$ is allowed to be nonzero.
Injectivity of one transition map is the required condition.

\begin{proposition}[Relative lifting]\label{prop:relative-lifting}
Let $R\ge2$ and suppose that
\[
 H^2(\mathcal G,T_R)\longrightarrow H^2(\mathcal G,T_{R-1})
\]
is injective. Let $j\ge R$, and suppose that a representation
\[
 \widetilde\rho_{j+R}:\mathcal G
 \longrightarrow\GG(\Z/p^{j+R}\Z)
\]
has been constructed whose reduction modulo $p^R$ is $\rho_R$.  Let
$\tau_{j+1}$ be its reduction modulo $p^{j+1}$.  Then $\tau_{j+1}$
admits a lift
\[
 \widetilde\rho_{j+R+1}:\mathcal G
 \longrightarrow\GG(\Z/p^{j+R+1}\Z).
\]
\end{proposition}

\begin{proof}
Consider the square-zero extension
\[
 \Z/p^{j+R+1}\Z\longrightarrow\Z/p^{j+1}\Z.
\]
By Lemma~\ref{lem:window-coeff}, its coefficient module for the
obstruction to lifting $\tau_{j+1}$ is $T_R$.  Denote the resulting
class by
\[
        o_R(\tau_{j+1})\in H^2(\mathcal G,T_R).
\]
Now replace the upper ring by the intermediate quotient
$\Z/p^{j+R}\Z$.  The corresponding kernel has width $R-1$, so its
coefficient module is $T_{R-1}$.  Naturality of obstruction classes
shows that the image of $o_R(\tau_{j+1})$ under
\[
 H^2(\mathcal G,T_R)\longrightarrow H^2(\mathcal G,T_{R-1})
\]
is precisely the obstruction to lifting $\tau_{j+1}$ to level
$p^{j+R}$.  But such a lift is already given by
$\widetilde\rho_{j+R}$.  Hence the image of $o_R(\tau_{j+1})$ is zero.
Injectivity forces $o_R(\tau_{j+1})=0$, and
Lemma~\ref{lem:obstruction} gives the required lift.
\end{proof}

\begin{remark}\label{rem:relative-not-compatible}
The new representation modulo $p^{j+R+1}$ is not required to reduce to
$\widetilde\rho_{j+R}$ modulo $p^{j+R}$. Only the mod $p^{j+1}$ representation is fixed. This is the relative feature which replaces ordinary
unobstructedness.
\end{remark}
\section{Cyclotomic descent and finite rational obstruction}\label{sec:cyclotomic}

We now specialize to $G=G_{\{p\}}$.
The purpose of this section is only to prove that the rational obstruction $H^2(G_{\{p\}},\Qp(m))$ can be nonzero for at most finitely many integers $m$.
\par We let $Y:=H^1(H,\Qp/\Zp)$ and $X:=Y^\vee$. Since $H$ acts trivially on
$\Qp/\Zp$, one has
\[
        H^1(H,\Qp/\Zp)
        =
        \Hom_{\mathrm{cts}}(H,\Qp/\Zp).
\]
Every homomorphism in this group factors through the maximal abelian pro-$p$ quotient
of $H$. Consequently, by Pontryagin duality, $X$ is naturally identified with
this maximal abelian pro-$p$ quotient.

\par
Equivalently, let $M_\infty$ denote the maximal abelian pro-$p$ extension of
$F_\infty$ contained in $\Q_{\{p\}}$. Then $M_\infty/F_\infty$ is unramified
at every finite prime not lying above $p$, and
\[
        X\simeq\Gal(M_\infty/F_\infty).
\]
Thus $X$ is the standard $p$-ramified Iwasawa module attached to the
cyclotomic extension $F_\infty/\Q$. The action of
$\Gamma=\Gal(F_\infty/\Q)$ on $H$ by conjugation induces an action on $X$,
making $X$ a compact module over the Iwasawa algebra
$\Lambda=\Zp[[\Gamma_1]]$, and more precisely over $\Lambda[\Delta_p]$. It is finitely generated over $\Lambda$; see \cite[Proposition~11.3.1]{NSW}. We have that
\begin{equation}\label{eq:weak-leopoldt-H}
 H^2(H,\Qp/\Zp)=0,
\end{equation}
cf. \cite[Theorem~10.3.25]{NSW}. Since $\chi|_H=1$, the same vanishing holds after every Tate twist.

For $m\in\Z$, recall that
\[
 e_m=\frac{1}{p-1}
 \sum_{\delta\in\Delta_p}\omega(\delta)^{-m}\delta
 \in\Zp[\Delta_p]
\]
and set
\[
 f_m=\gamma^{-1}-u^{-m}\in\Lambda.
\]
We note that $f_m=T-(u^{-m}-1)$. In particular $(f_m)$ is a height-one prime and $\Lambda/(f_m)\simeq\Zp$.

\begin{proposition}\label{prop:cyclotomic-specialization}
For every $m\in\Z$ there is a natural isomorphism
\begin{equation}\label{eq:H2-specialization}
 H^2\!\left(G_{\{p\}},(\Qp/\Zp)(m)\right)^\vee
 \simeq e_mX[f_m].
\end{equation}
\end{proposition}

\begin{proof}
Put $A:=\Qp/\Zp$ and $A_m:=A(m)$. Recall that
$H=\Gal(\Q_{\{p\}}/F_\infty)$ and
$\Gamma=\Gal(F_\infty/\Q)$. Since $F_\infty$ contains all
$p$-power roots of unity, the cyclotomic character is trivial on
$H$. Thus $H$ acts trivially on the Tate-twist factor of $A_m$.
Consequently, for every $q\geq0$, there is an identification
$H^q(H,A_m)=H^q(H,A)(m)$ as $\Gamma$-modules. 

We first apply the Hochschild--Serre spectral sequence to the exact
sequence
\[
 1\longrightarrow H\longrightarrow G_{\{p\}}
 \longrightarrow\Gamma\longrightarrow1.
\]
It has second page
\[
 E_2^{a,b}
 =
 H^a\bigl(\Gamma,H^b(H,A_m)\bigr)
 \Longrightarrow
 H^{a+b}(G_{\{p\}},A_m).
\]
We are interested in total degree $2$. There are therefore three
potential terms, namely $E_2^{0,2}$, $E_2^{1,1}$ and $E_2^{2,0}$. The first of these vanishes. Indeed, weak Leopoldt \eqref{eq:weak-leopoldt-H} for the
cyclotomic extension gives $H^2(H,A)=0$. Since the Tate twist does
not change the action of $H$, it follows that $H^2(H,A_m)=0$, and
hence $E_2^{0,2}=0$.

The third term also vanishes. Recall that
$\Gamma=\Delta_p\times\Gamma_1$, where $\Gamma_1\simeq\Zp$ and
$|\Delta_p|=p-1$. The group $\Gamma_1$ has $p$-cohomological
dimension one, while a finite group of order prime to $p$ has
$p$-cohomological dimension zero. Hence
$\operatorname{cd}_p(\Gamma)=1$. Since $A_m$ is a discrete
$p$-primary module, this gives
$E_2^{2,0}=H^2(\Gamma,A_m)=0$.

Thus the only possibly nonzero term in total degree $2$ is
\[
 E_2^{1,1}
 =
 H^1\bigl(\Gamma,H^1(H,A_m)\bigr).
\]
Since $H^1(H,A_m)=Y(m)$, this term is
$H^1(\Gamma,Y(m))$. Moreover, no differential can enter
$E_2^{1,1}$, since its possible source would have negative first
index, and no differential can leave it: the first possible target
is $E_2^{3,0}$, which vanishes because
$\operatorname{cd}_p(\Gamma)=1$. Therefore
$E_\infty^{1,1}=E_2^{1,1}$. As all the other graded pieces of the
Hochschild--Serre filtration on $H^2(G_{\{p\}},A_m)$ vanish, we
obtain a canonical isomorphism
\begin{equation}\label{eq:HS-H2}
 H^2(G_{\{p\}},A_m)
 \simeq H^1(\Gamma,Y(m)).
\end{equation}

We next compute the group on the right. Apply Hochschild--Serre to
the exact sequence $1\to\Gamma_1\to\Gamma\to\Delta_p\to1$. Since
$|\Delta_p|=p-1$ is prime to $p$, one has
$H^a(\Delta_p,D)=0$ for every $a>0$ and every discrete
$p$-primary $\Delta_p$-module $D$. Indeed, multiplication by
$|\Delta_p|$ is an automorphism of $D$, so the usual averaging
operator makes the invariants functor exact. It follows that
\[
 H^1(\Gamma,Y(m))
 \simeq H^1(\Gamma_1,Y(m))^{\Delta_p}.
\]

Let $\gamma$ be the fixed topological generator of $\Gamma_1$ and
put $u=\chi(\gamma)$. For any discrete $p$-primary
$\Gamma_1$-module $D$, continuous cohomology of
$\Gamma_1\simeq\Zp$ gives
$H^1(\Gamma_1,D)=D/(\gamma-1)D$. On the twisted module $Y(m)$,
the element $\gamma$ acts as $u^m\gamma$, where the second
$\gamma$ denotes its original action on $Y$. Consequently,
\[
 H^1(\Gamma_1,Y(m))
 \simeq Y/(u^m\gamma-1)Y.
\]
For brevity, denote this quotient by $Q_m$. Thus
$H^1(\Gamma,Y(m))=Q_m^{\Delta_p}$.

We now dualize. Since $X=Y^\vee$, the Pontryagin dual of $Q_m$ is
the annihilator in $X$ of $(u^m\gamma-1)Y$. Under the usual
contragredient action on the Pontryagin dual, the adjoint of
$\gamma$ acting on $Y$ is $\gamma^{-1}$ acting on $X$. Therefore
\[
 Q_m^\vee
 =
 \ker\bigl(u^m\gamma^{-1}-1:X\longrightarrow X\bigr).
\]
Multiplying the operator by the unit $u^{-m}$ does not change its
kernel. Hence
\[
 Q_m^\vee
 =
 \ker(\gamma^{-1}-u^{-m}:X\longrightarrow X)
 =
 X[f_m],
\]
where, by definition, $f_m=\gamma^{-1}-u^{-m}$.

It remains to keep track of the action of the finite group
$\Delta_p$. This is the point at which the idempotent $e_m$ enters.
For $\delta\in\Delta_p$, the action on the twisted module $Y(m)$ is
the original action on $Y$ multiplied by
$\chi(\delta)^m=\omega(\delta)^m$. Therefore an element of $Q_m$
is fixed by $\Delta_p$ after twisting precisely when, with respect
to the original action on $Y$, it belongs to the
$\omega^{-m}$-eigenspace. In terms of the standard idempotents,
this says
\[
 Q_m^{\Delta_p}=e_{-m}Q_m.
\]

Pontryagin duality reverses the finite character. More explicitly,
if $y$ transforms under $\Delta_p$ by $\omega^{-m}$, then the
contragredient action on a functional pairing nontrivially with
$y$ is by $\omega^m$. Equivalently, under the involution
$\delta\mapsto\delta^{-1}$ on $\Zp[\Delta_p]$, the idempotent
$e_{-m}$ is carried to $e_m$. It follows that
\[
 (Q_m^{\Delta_p})^\vee
 =
 (e_{-m}Q_m)^\vee
 \simeq e_mQ_m^\vee.
\]
Using $Q_m^\vee=X[f_m]$ and the fact that $f_m$ commutes with the
$\Delta_p$-action, we finally obtain
\[
 H^1(\Gamma,Y(m))^\vee
 \simeq e_mX[f_m].
\]
Combining this with \eqref{eq:HS-H2} proves
\eqref{eq:H2-specialization}.
\end{proof}

\begin{lemma}\label{lem:compact-H3-vanishing}
Let $p$ be odd and let $T$ be a finite free $\Zp$-module equipped with
a continuous action of $G_{\{p\}}$. Then
\[
        H^3(G_{\{p\}},T)=0.
\]
\end{lemma}

\begin{proof}
Write $T=\varprojlim_r T/p^rT$. By \cite[Theorem~2.7.5]{NSW}, continuous cohomology with compact
coefficients gives a natural exact sequence
\[
 0\longrightarrow
 \varprojlim\nolimits_r^{1}
 H^2(G_{\{p\}},T/p^rT)
 \longrightarrow
 H^3(G_{\{p\}},T)
 \longrightarrow
 \varprojlim_r
 H^3(G_{\{p\}},T/p^rT)
 \longrightarrow0.
\]
Here $\varprojlim^{1}$ denotes the first right-derived functor of the inverse limit. For every $r\geq1$, the module $T/p^rT$ is a finite discrete
$p$-primary $G_{\{p\}}$-module, and
$H^2(G_{\{p\}},T/p^rT)$ is finite; see
\cite[Theorem~8.3.19]{NSW}. For each fixed $r$, the images of the
groups at higher levels in $H^2(G_{\{p\}},T/p^rT)$ form a descending
chain of subgroups of a finite group and hence eventually stabilize.
Thus the inverse system satisfies the Mittag--Leffler condition, and
\cite[Proposition~2.7.4]{NSW} gives
\[
        \varprojlim\nolimits_r^{1}
        H^2(G_{\{p\}},T/p^rT)=0.
\]

Since $p$ is odd, $\operatorname{cd}_p(G_{\{p\}})=2$. Therefore
\[
        H^3(G_{\{p\}},T/p^rT)=0
\]
for every $r\geq1$. Both outer terms in the preceding exact sequence
vanish, and hence
\[
        H^3(G_{\{p\}},T)=0.
\]
\end{proof}

\begin{lemma}\label{lem:rational-rank}
For every $m\in\Z$, the groups $H^i(G_{\{p\}},\Zp(m))$ are finitely generated $\Zp$-modules for $i=1,2$, and the canonical maps
\[
 H^i(G_{\{p\}},\Zp(m))\otimes_{\Zp}\Qp
 \longrightarrow
 H^i(G_{\{p\}},\Qp(m))
 \qquad \text{for}\quad i=1,2
\]
are isomorphisms.  In particular, the degree-two map
\[
 H^2(G_{\{p\}},\Zp(m))\otimes_{\Zp}\Qp
 \longrightarrow
 H^2(G_{\{p\}},\Qp(m))
\]
is an isomorphism.  Moreover
\[
 \dim_{\Qp}H^2(G_{\{p\}},\Qp(m))
 =
 \operatorname{corank}_{\Zp}
 H^2\!\left(G_{\{p\}},(\Qp/\Zp)(m)\right).
\]
\end{lemma}

\begin{proof}
By \cite[Theorem~8.3.19]{NSW},
$H^2(G_{\{p\}},\Fp(m))$ is finite. By
Lemma~\ref{lem:compact-H3-vanishing}, we have
\[
        H^3(G_{\{p\}},\Zp(m))=0.
\]
Hence the coefficient sequence
\[
0\longrightarrow \Zp(m)
\xrightarrow{p}
\Zp(m)
\longrightarrow
\Fp(m)
\longrightarrow 0
\]
gives
\[
H^2(G_{\{p\}},\Zp(m))/pH^2(G_{\{p\}},\Zp(m))
\simeq
H^2(G_{\{p\}},\Fp(m)).
\]
Thus the reduction modulo $p$ of the compact $\Zp$-module
$H^2(G_{\{p\}},\Zp(m))$ is finite. By the topological Nakayama
lemma, it follows that $H^2(G_{\{p\}},\Zp(m))$ is finitely
generated over $\Zp$. The same coefficient sequence gives an injection
\[
H^1(G_{\{p\}},\Zp(m))/pH^1(G_{\{p\}},\Zp(m))
\hookrightarrow
H^1(G_{\{p\}},\Fp(m)).
\]
The group on the right is finite by
\cite[Theorem~8.3.19]{NSW}, so another application of the topological
Nakayama lemma gives finite generation of
$H^1(G_{\{p\}},\Zp(m))$.

Continuous cohomology in this finite cohomological-dimension setting
commutes with passage from a $\Zp$-lattice to its associated
$\Qp$-vector space. Equivalently, passing to the direct limit through
the lattices
\[
 \Zp\cdot p^{-k}(m)\hookrightarrow\Zp\cdot p^{-k-1}(m)
\]
gives
\[
 H^2(G_{\{p\}},\Zp(m))\otimes_{\Zp}\Qp
 \simeq H^2(G_{\{p\}},\Qp(m)).
\]
The identical argument in degree one gives
\[
 H^1(G_{\{p\}},\Zp(m))\otimes_{\Zp}\Qp
 \simeq H^1(G_{\{p\}},\Qp(m)).
\]
For this compatibility, see also \cite[Chapter~2, \S7]{NSW}.

For the corank assertion, use
\[
 0\longrightarrow\Zp(m)
 \longrightarrow\Qp(m)
 \longrightarrow(\Qp/\Zp)(m)
 \longrightarrow0.
\]
By Lemma~\ref{lem:compact-H3-vanishing}, the resulting long exact
sequence ends with
\[
 H^2(G_{\{p\}},\Zp(m))
 \longrightarrow H^2(G_{\{p\}},\Qp(m))
 \longrightarrow H^2(G_{\{p\}},(\Qp/\Zp)(m))
 \longrightarrow0.
\]
The image of the first term is a full $\Zp$-lattice in the
finite-dimensional $\Qp$-space in the middle. The quotient therefore
has $\Zp$-corank equal to the dimension of that $\Qp$-space.
\end{proof}

The next commutative-algebra observation is what makes the exceptional set finite.

\begin{lemma}\label{lem:associated-prime}
Let $M$ be a finitely generated
$\Lambda=\Zp[[T]]$-module, and let $f=T-a$ with $a\in p\Zp$.
If $M[f]$ has positive $\Zp$-rank, then $(f)$ is an associated
prime of $M$.
\end{lemma}

\begin{proof}
Since $M[f]$ has positive $\Zp$-rank, there exists
$x\in M[f]$ which is not $\Zp$-torsion. In particular, $x$ is not
annihilated by any power of $p$.

Since $fx=0$, the action of $\Lambda$ on the cyclic submodule
$\Lambda x$ factors through $\Lambda/(f)$. Because $a\in p\Zp$,
evaluation at $T=a$ gives an isomorphism
$\Lambda/(f)\simeq\Zp$. Thus $\Lambda x$ may be regarded as a
$\Zp$-module generated by $x$.

We claim that $\Ann_\Lambda(x)=(f)$. The inclusion
$(f)\subseteq\Ann_\Lambda(x)$ is immediate from $fx=0$. Conversely,
suppose that $g\in\Lambda$ satisfies $gx=0$, and let $\bar g$ denote
its image in $\Lambda/(f)\simeq\Zp$. Then $\bar g x=0$. If
$\bar g\neq0$, write $\bar g=p^r u$ with $r\geq0$ and
$u\in\Zp^\times$. Since multiplication by $u$ is an automorphism of
the $\Zp$-module $\Lambda x$, the equality $\bar g x=0$ implies
$p^r x=0$, contradicting the choice of $x$. Hence $\bar g=0$, so
$g\in(f)$.

Therefore $\Ann_\Lambda(x)=(f)$. Since
$\Lambda/(f)\simeq\Zp$ is an integral domain, $(f)$ is prime, and
hence, by definition, $(f)\in\Ass_\Lambda(M)$.
\end{proof}

\begin{proposition}\label{prop:finite-exceptional}
There is a finite set $\cE_p\subset\Z$ such that
\begin{equation}\label{eq:rational-H2-good}
 H^2(G_{\{p\}},\Qp(m))=0
\end{equation}
for every $m\notin\cE_p$.  Moreover $H^2(G_{\{p\}},\Qp)=0$.
\end{proposition}

\begin{proof}
By Proposition~\ref{prop:cyclotomic-specialization} and Lemma~\ref{lem:rational-rank}, a nonzero rational group in \eqref{eq:rational-H2-good} implies that $e_mX[f_m]$ has positive $\Zp$-rank.  Lemma~\ref{lem:associated-prime} therefore shows that $(f_m)$ is an associated prime of the finitely generated module $e_mX$.  Note that $e_m$ depends only on $m$ modulo $p-1$. Thus there are $p-1$ modules $e_m X$ to consider and for each module, the set $\Ass_\Lambda(e_mX)$ is finite; see \cite[Theorem~6.5(i)]{MatsumuraCRT}. Finally the ideals $(f_m)$ are pairwise distinct because $u$ has infinite order in $1+p\Zp$.  Thus only finitely many integers $m$ can occur.

The trivial twist requires a separate argument. Since $\Qp$ has trivial
$G_{\{p\}}$-action, one has $H^0(G_{\{p\}},\Qp)=\Qp$. Moreover,
\[
        H^1(G_{\{p\}},\Qp)
        =
        \Hom_{\mathrm{cts}}(G_{\{p\}},\Qp).
\]
By global class field theory, the maximal abelian pro-$p$ quotient of
$G_{\{p\}}$ has $\Zp$-rank one. Indeed, its unique infinite
$\Zp$-extension is the cyclotomic $\Zp$-extension of $\Q$. Hence
$\dim_{\Qp}H^1(G_{\{p\}},\Qp)=1$.

We now apply the global Euler--Poincar\'e characteristic formula for
$p$-adic representations; see \cite[Lemma~9.7]{Kisin2003}. For a
finite-dimensional $\Qp$-representation $V$ of $G_{\{p\}}$, one has
\[
 \sum_{i=0}^{2}(-1)^i
 \dim_{\Qp}H^i(G_{\{p\}},V)
 =
 \dim_{\Qp}V^{c=1}-\dim_{\Qp}V,
\]
where $c$ denotes complex conjugation.

For the trivial representation $V=\Qp$, complex conjugation acts
trivially, so the right-hand side is zero. Since
$H^0(G_{\{p\}},\Qp)$ and $H^1(G_{\{p\}},\Qp)$ both have dimension
one, it follows that
\[
        H^2(G_{\{p\}},\Qp)=0.
\]
\end{proof}

\section{The integral toral representation}\label{sec:toral}

Let $L=X_*(\TT)$.  By hypothesis there is
$\lambda_0\in L$ such that
$\langle\alpha,\lambda_0\rangle$ is odd for every
$\alpha\in\Delta$. Put $\mathcal B=\Delta\cup(-\Delta)$.

\begin{lemma}[Choice of cocharacter]\label{lem:weights}
Let $\cE\subset\Z$ be finite. There exists a cocharacter
$\lambda\in X_*(\TT)$ such that:
\begin{enumerate}[label=\textup{(\arabic*)}]
\item $\langle\alpha,\lambda\rangle$ is odd for every
      $\alpha\in\Delta$;
\item $\langle\beta,\lambda\rangle\notin\cE$ for every
      $\beta\in\Phi$;
\item the integers $\langle\alpha,\lambda\rangle$, for
      $\alpha\in\mathcal B$, are pairwise distinct;
\item if $\dim Z(\GG)=1$, the image
      $\lambda^{\mathrm{ab}}\in X_*(\GG/\Gder)$ is nonzero.
\end{enumerate}
\end{lemma}

\begin{proof}
Put $L=X_*(\TT)$ and begin with the coset
$\lambda_0+2L$, where $\lambda_0$ is the odd cocharacter fixed above.
Every element of this coset has the same simple-root pairings modulo
$2$ as $\lambda_0$. Hence condition~(1) holds automatically.

We next impose conditions~(2) and~(3). Fix a root $\beta\in\Phi$ and
an integer $e\in\cE$. The condition
$\langle\beta,\lambda\rangle=e$ defines a proper affine hyperplane in
$L\otimes_{\Z}\Q$. Indeed, the root $\beta$ is a nonzero character of
$\TT$, and hence the linear functional
$\lambda\mapsto\langle\beta,\lambda\rangle$ is nonzero on
$L\otimes_{\Z}\Q$.

Likewise, if $\alpha,\beta\in\mathcal B$ are distinct, then the
condition
$\langle\alpha,\lambda\rangle=\langle\beta,\lambda\rangle$ is
equivalent to
$\langle\alpha-\beta,\lambda\rangle=0$. Since $\alpha\neq\beta$ as
characters of $\TT$, the character $\alpha-\beta$ is nonzero, and this
again defines a proper hyperplane in $L\otimes_{\Z}\Q$.

There are only finitely many roots, finitely many elements of $\cE$,
and finitely many pairs of elements of $\mathcal B$. Consequently the
cocharacters which fail either condition~(2) or condition~(3) lie in
a finite union of proper affine hyperplanes. A finite union of proper
affine hyperplanes cannot contain the full-rank lattice coset
$\lambda_0+2L$. We may therefore choose $\lambda$ in this coset which
satisfies conditions~(1)--(3). In fact, there are infinitely many such
choices.

Suppose now that $\dim Z(\GG)=1$. Since $\GG$ is split, the connected
center $Z(\GG)^0$ is a one-dimensional split torus. Choose a nonzero
cocharacter $\zeta\in X_*(Z(\GG)^0)$. The natural map
$Z(\GG)^0\to\GG/\Gder$ has finite kernel, namely
$Z(\GG)^0\cap\Gder$. It follows that the image
$\zeta^{\mathrm{ab}}\in X_*(\GG/\Gder)$ is nonzero. Since every root
of $\GG$ is trivial on the center, replacing $\lambda$ by
$\lambda+2N\zeta$ does not change any of the root pairings
$\langle\beta,\lambda\rangle$. Thus conditions~(1)--(3) remain
unchanged.

The lattice $X_*(\GG/\Gder)$ has rank one. Hence
$\lambda^{\mathrm{ab}}+2N\zeta^{\mathrm{ab}}$ can vanish for at most
one integer $N$. Choosing any other $N$ gives condition~(4), while
preserving conditions~(1)--(3).
\end{proof}
Apply Lemma~\ref{lem:weights} with $\cE=\cE_p\cup\{0\}$. Fix the resulting cocharacter $\lambda$ and define
\begin{equation}\label{eq:rho-infty}
 \rho_\infty=\lambda\circ\chi:
 G_{\{p\}}\longrightarrow\TT(\Zp)\subset\GG(\Zp).
\end{equation}
For every root $\beta$ put
\[
        d_\beta=\langle\beta,\lambda\rangle.
\]
Then $d_{-\beta}=-d_\beta$, every simple-root weight is odd, and every
root weight lies outside $\cE_p$.  Put $T=\Ad\rho_\infty(\frakg_{\Zp})$ and $T_r=T/p^rT$. By \eqref{eq:Ad-toral-decomposition}, we have that
\begin{equation}\label{eq:T-decomposition}
 T=\frakt_{\Zp}\oplus
 \bigoplus_{\beta\in\Phi}\Zp(d_\beta)X_\beta.
\end{equation}

\begin{proposition}\label{prop:H2-finite}
The group $H^2(G_{\{p\}},T)$ is finite.
\end{proposition}

\begin{proof}
By \eqref{eq:T-decomposition}, we have a decomposition of
$G_{\{p\}}$-modules
\[
        T
        =
        \frakt_{\Zp}
        \oplus
        \bigoplus_{\beta\in\Phi}
        \Zp(d_\beta)X_\beta,
\]
where the action on $\frakt_{\Zp}$ is trivial. Thus
$\frakt_{\Zp}$ is a finite direct sum of copies of $\Zp$.

Applying Lemma~\ref{lem:rational-rank} to the trivial twist and to
each of the finitely many twists $d_\beta$, we find that
$H^2(G_{\{p\}},T)$ is a finitely generated $\Zp$-module and that
\[
        H^2(G_{\{p\}},T)\otimes_{\Zp}\Qp
        \simeq
        H^2\!\left(G_{\{p\}},T\otimes_{\Zp}\Qp\right).
\]

After tensoring \eqref{eq:T-decomposition} with $\Qp$, the toral
summand is a finite direct sum of copies of the trivial representation
$\Qp$, while the root summand indexed by $\beta$ is
$\Qp(d_\beta)X_\beta$. Proposition~\ref{prop:finite-exceptional}
gives
\[
        H^2(G_{\{p\}},\Qp)=0
\]
and, since $d_\beta\notin\cE_p$ for every $\beta\in\Phi$,
\[
        H^2(G_{\{p\}},\Qp(d_\beta))=0.
\]
Hence
\[
        H^2\!\left(
        G_{\{p\}},T\otimes_{\Zp}\Qp
        \right)=0.
\]
It follows that
\[
        H^2(G_{\{p\}},T)\otimes_{\Zp}\Qp=0.
\]
Therefore $H^2(G_{\{p\}},T)$ is a finitely generated torsion
$\Zp$-module, and hence is finite.
\end{proof}

Choose $e\ge0$ such that
\begin{equation}\label{eq:e-kills-H2}
        p^e H^2(G_{\{p\}},T)=0.
\end{equation}

\begin{proposition}\label{prop:H2-stability}
For every $r\ge e+1$, the natural reduction map
\[
 H^2(G_{\{p\}},T_r)
 \longrightarrow
 H^2(G_{\{p\}},T_{r-1})
\]
is an isomorphism.
\end{proposition}

\begin{proof}
It follows from Lemma \ref{lem:compact-H3-vanishing} that $H^3(G_{\{p\}},T)=0$. Consider the short exact sequence
\[
        0\longrightarrow T
        \xrightarrow{p^r}
        T
        \longrightarrow T_r
        \longrightarrow0.
\]
Since $H^3(G_{\{p\}},T)=0$, the corresponding long exact sequence
contains
\[
 H^2(G_{\{p\}},T)
 \xrightarrow{p^r}
 H^2(G_{\{p\}},T)
 \longrightarrow
 H^2(G_{\{p\}},T_r)
 \longrightarrow0.
\]
It follows that
\begin{equation}\label{eq:H2-quotient}
 H^2(G_{\{p\}},T_r)
 \simeq
 H^2(G_{\{p\}},T)/
 p^rH^2(G_{\{p\}},T).
\end{equation}

The reduction map $T_r\to T_{r-1}$ is compatible with the two short
exact sequences defining $T_r$ and $T_{r-1}$. Under the
identifications \eqref{eq:H2-quotient}, the induced map on
cohomology is therefore the natural quotient map
\[
 H^2(G_{\{p\}},T)/
 p^rH^2(G_{\{p\}},T)
 \longrightarrow
 H^2(G_{\{p\}},T)/
 p^{r-1}H^2(G_{\{p\}},T).
\]

By \eqref{eq:e-kills-H2}, multiplication by $p^e$ annihilates
$H^2(G_{\{p\}},T)$. If $r\ge e+1$, then both $p^r$ and
$p^{r-1}$ annihilate this group. Hence both the source and target of
the preceding quotient map identify canonically with
$H^2(G_{\{p\}},T)$, and under these identifications the transition
map is the identity. It is therefore an isomorphism.
\end{proof}

Fix from now on an integer
\begin{equation}\label{eq:R-choice}
        R\ge\max\{2,e+1\}.
\end{equation}
Thus
\[
 H^2(G_{\{p\}},T_R)\longrightarrow H^2(G_{\{p\}},T_{R-1})
\]
is injective, which is the only obstruction-theoretic property needed
in the relative tower.

We next construct the root-valued cohomology classes used at the
initial finite level.

\begin{lemma}\label{lem:odd-H1}
Let $d$ be odd and suppose
$H^2(G_{\{p\}},\Qp(d))=0$. Then
\[
        \dim_{\Qp}H^1(G_{\{p\}},\Qp(d))=1.
\]
Moreover $H^1(G_{\{p\}},\Zp(d))$ has $\Zp$-rank one and contains an
integral class mapping to a nonzero rational class.
\end{lemma}

\begin{proof}
Since $d\neq0$,
$H^0(G_{\{p\}},\Qp(d))=0$.  Complex conjugation acts on
$\Qp(d)$ by $(-1)^d=-1$.  The global Euler--Poincare characteristic
formula gives
\[
 \dim H^0-\dim H^1+\dim H^2=-1.
\]
The first and third terms vanish, and therefore $\dim H^1=1$.
The degree-one assertion of Lemma~\ref{lem:rational-rank} identifies
$H^1(G_{\{p\}},\Zp(d))\otimes\Qp$ with this one-dimensional rational
space.  The integral group therefore has rank one, and any element
mapping to a nonzero rational vector gives the required class.
\end{proof}

\begin{lemma}\label{lem:restriction}
If $d\neq0$, the restriction map
\[
 H^1(G_{\{p\}},\Qp(d))
 \longrightarrow H^1(H,\Qp(d))
\]
is injective.
\end{lemma}

\begin{proof}
Inflation--restriction for
\[
 1\longrightarrow H\longrightarrow G_{\{p\}}
 \xrightarrow{\chi}\Zp^\times\longrightarrow1
\]
identifies the kernel with $H^1(\Zp^\times,\Qp(d))$.  Write
$\Zp^\times=\Delta_p\times(1+p\Zp)$.  Since $|\Delta_p|$ is invertible
in $\Qp$, its positive-degree cohomology vanishes.  A topological
generator of $1+p\Zp$ acts on $\Qp(d)$ by $u^d\neq1$, so
\[
 H^1(1+p\Zp,\Qp(d))
 =\Qp/(u^d-1)\Qp=0.
\]
Hence the kernel of restriction is zero.
\end{proof}

For each $\alpha\in\mathcal B$, the integer $d_\alpha$ is odd and
lies outside $\cE_p$.  Choose an integral cocycle
\begin{equation}\label{eq:z-alpha-cocycles}
        z_\alpha\in
        Z^1(G_{\{p\}},\Zp(d_\alpha))
\end{equation}
representing a nonzero rational class.  By
Lemma~\ref{lem:restriction}, its restriction to $H$ is nonzero.
Because $H$ acts trivially on the coefficient module, the
restriction is an ordinary continuous homomorphism $H\longrightarrow\Zp$. Its image is a nonzero closed subgroup of $\Zp$, hence equals
$p^{s_\alpha}\Zp$ for a unique $s_\alpha\ge0$.  Set
\begin{equation}\label{eq:s-choice}
        s=\max_{\alpha\in\mathcal B}s_\alpha.
\end{equation}

\section{The initial finite-level lift}\label{sec:initial}

We now construct a representation modulo a sufficiently large power of
$p$ whose image already contains the positive and negative simple-root
directions. Define
\begin{equation}\label{eq:F-cocycle}
        F(g)=\sum_{\alpha\in\mathcal B}z_\alpha(g)X_\alpha.
\end{equation}
Each summand lies in the root line
$\Zp(d_\alpha)X_\alpha$, so $F$ is a $1$-cocycle with values in
$T=\Ad\rho_\infty(\frakg_{\Zp})$.

Choose $v=1+p$ and choose $\sigma\in G_{\{p\}}$ with
$\chi(\sigma)=v$. For $\alpha\in\mathcal B$ put
\[
        \lambda_\alpha=v^{d_\alpha}\in\Zp^\times.
\]
By the choice of $\lambda$, the integers $d_\alpha$ for
$\alpha\in\mathcal B$ are pairwise distinct; because $v$ has infinite
order, the $p$-adic units $\lambda_\alpha$ are pairwise distinct.
Conjugation by $\rho_\infty(\sigma)$ acts on $X_\alpha$ as
multiplication by $\lambda_\alpha$.

There are finitely many pairs $\alpha\neq\beta$, so define
\begin{equation}\label{eq:m-separation}
 m=1+
 \max_{\substack{\alpha,\beta\in\mathcal B\\\alpha\neq\beta}}
 v_p(\lambda_\alpha-\lambda_\beta).
\end{equation}
Then
\begin{equation}\label{eq:lambda-separation}
 v_p(\lambda_\alpha-\lambda_\beta)<m
 \qquad(\alpha\neq\beta).
\end{equation}
If $A\ge m$, the eigenvalues remain pairwise distinct modulo $p^A$.
Set
\begin{equation}\label{eq:C-projector}
        C=m(|\mathcal B|-1).
\end{equation}

\begin{lemma}\label{lem:projector}
Let $A\ge m$ and let $M\subset T_A$ be a
$G_{\{p\}}$-stable $\Z/p^A\Z$-submodule. Suppose
\[
        x=\sum_{\beta\in\mathcal B}c_\beta X_\beta\in M
\]
is supported on the simple-root lines. Fix $\alpha\in\mathcal B$ and
suppose $c_\alpha\neq0$ with $v_p(c_\alpha)\le s_0$.  If
$s_0+C<A$, then
\[
        p^{s_0+C}X_\alpha\in M.
\]
\end{lemma}

\begin{proof}
Let
\[
        D=\Ad\rho_A(\sigma)
\]
and define
\[
        P_\alpha(X)=
        \prod_{\substack{\beta\in\mathcal B\\\beta\neq\alpha}}
        (X-\lambda_\beta).
\]
Because $M$ is stable under the Galois action, it is stable under $D$
and hence under every polynomial in $D$ with coefficients in
$\Z/p^A\Z$.  On the root vectors,
$D(X_\beta)=\lambda_\beta X_\beta$.  Therefore
$P_\alpha(D)X_\beta=0$ for $\beta\neq\alpha$, while
\[
 P_\alpha(D)X_\alpha
 =\prod_{\beta\neq\alpha}
   (\lambda_\alpha-\lambda_\beta)X_\alpha.
\]
It follows that
\[
 P_\alpha(D)x
 =c_\alpha
  \prod_{\beta\neq\alpha}
  (\lambda_\alpha-\lambda_\beta)X_\alpha.
\]
Write $t=v_p(c_\alpha)$ and
\[
 r_\alpha=
 v_p\!\left(
  \prod_{\beta\neq\alpha}
  (\lambda_\alpha-\lambda_\beta)
 \right).
\]
There are $|\mathcal B|-1$ factors, each of valuation strictly less
than $m$, hence $r_\alpha<C$.  Consequently
$t+r_\alpha<s_0+C<A$, so the displayed coefficient is nonzero modulo
$p^A$ and may be written
$p^{t+r_\alpha}u_\alpha$ with $u_\alpha$ a unit.  Multiplying by
$u_\alpha^{-1}p^{s_0+C-t-r_\alpha}$ gives the desired element.
\end{proof}

Choose $A$ so large that
\begin{equation}\label{eq:A-choice}
        A\ge m,
        \qquad
        A>s+C+R+2.
\end{equation}
If $\dim Z(\GG)=1$, we also enlarge $A$ so that the reduction of the
nonzero central character $\lambda^{\mathrm{ab}}\circ\chi$ modulo
$p^A$ already has nontrivial pro-$p$ image. The ideal $p^A\Z/p^{2A}\Z$ is square-zero. Therefore smoothness gives
a canonical additive isomorphism
\begin{equation}\label{eq:kappa-A}
 \kappa_A:T_A
 \xrightarrow{\sim}
 \ker\!\left(\GG(\Z/p^{2A}\Z)
       \to\GG(\Z/p^A\Z)\right).
\end{equation}
We record explicitly how this square-zero identification interacts with
congruence depth.

\begin{lemma}\label{lem:kappa-depth}
Let $0\le t<A$ and let $X\in\frakg_{\Zp}$ have nonzero reduction
$\bar X\in\frakg_{\Fp}$.  Then
$\kappa_A(p^tX\bmod p^A)$ lies in the image of the congruence subgroup
$U_{A+t}(\GG)$ modulo $p^{2A}$, and its leading term at depth $A+t$
is $\bar X$.
\end{lemma}

\begin{proof}
This is the functorial description of the tangent kernel for the
square-zero ideal $p^A\Z/p^{2A}\Z$.  After choosing a faithful
integral representation $\GG\hookrightarrow\GL_N$, the element
$\kappa_A(p^tX)$ has the form
\[
        1+p^{A+t}\widetilde X
        \pmod{p^{2A}},
\]
where $\widetilde X\bmod p=\bar X$.  Hence it is trivial modulo
$p^{A+t}$, nontrivial in general modulo $p^{A+t+1}$, and its image in
$U_{A+t}(\GG)/U_{A+t+1}(\GG)$ is exactly $\bar X$.  The conclusion
is independent of the chosen faithful representation because the
leading term is the intrinsic tangent vector.
\end{proof}

We define
\begin{equation}\label{eq:rho-2A}
 \rho_{2A}^*(g)
 =\kappa_A(F(g)\bmod p^A)\,\rho_{2A}(g).
\end{equation}
This formula is the intrinsic group-scheme replacement for
$(1+p^AF(g))\rho_\infty(g)$ in the matrix case.

\begin{lemma}\label{lem:initial-twist-rep}
The map
\[
        \rho_{2A}^*:G_{\{p\}}
        \longrightarrow\GG(\Z/p^{2A}\Z)
\]
is a continuous representation and its reduction modulo $p^A$ is
$\rho_A$.
\end{lemma}

\begin{proof}
Continuity is immediate.  The cocycle relation for $F$ is
\[
 F(gh)=F(g)+\Ad(\rho_\infty(g))F(h).
\]
Modulo $p^A$ the same identity holds in $T_A$.  The multiplication law
in the square-zero kernel and the conjugation formula give
\begin{align*}
 \rho_{2A}^*(g)\rho_{2A}^*(h)
 &=\kappa_A(F(g))\rho_{2A}(g)
   \kappa_A(F(h))\rho_{2A}(h)\\
 &=\kappa_A\!\left(
   F(g)+\Ad(\rho_A(g))F(h)
   \right)\rho_{2A}(gh)\\
 &=\kappa_A(F(gh))\rho_{2A}(gh)\\
 &=\rho_{2A}^*(gh).
\end{align*}
No quadratic error remains because the kernel is square-zero.  Finally
$\kappa_A(F(g))$ lies in the kernel of reduction to level $A$, so the
reduction of $\rho_{2A}^*$ is $\rho_A$.
\end{proof}

Put $K_A=\ker\rho_A$. Since $\rho_{2A}^*$ reduces to $\rho_A$, its
image on $K_A$ lies in the square-zero congruence kernel.  Hence there
is a unique additive subgroup $M_A\subset T_A$ such that
\begin{equation}\label{eq:def-MA}
        \rho_{2A}^*(K_A)=\kappa_A(M_A).
\end{equation}
An additive subgroup of the finite $\Z/p^A\Z$-module $T_A$ is stable
under integer multiplication and is therefore a
$\Z/p^A\Z$-submodule.

\begin{lemma}\label{lem:M-stable}
The submodule $M_A\subset T_A$ is stable under the adjoint action of
$G_{\{p\}}$ through $\rho_A$.
\end{lemma}

\begin{proof}
Take $X\in M_A$ and choose $h\in K_A$ with
$\rho_{2A}^*(h)=\kappa_A(X)$. For $g\in G_{\{p\}}$, normality of
$K_A$ gives $ghg^{-1}\in K_A$. Thus
$\rho_{2A}^*(ghg^{-1})=\kappa_A(Y)$ for some $Y\in M_A$. On the other
hand,
\[
 \rho_{2A}^*(ghg^{-1})
 =\rho_{2A}^*(g)\kappa_A(X)\rho_{2A}^*(g)^{-1}.
\]
Conjugation on the square-zero kernel depends only on the reduction of
$\rho_{2A}^*(g)$ modulo $p^A$, which is $\rho_A(g)$. Hence
\[
        Y=\Ad(\rho_A(g))X.
\]
Therefore $M_A$ is Galois-stable.
\end{proof}

We can now isolate every positive and negative simple root.

\begin{proposition}\label{prop:initial-roots}
Put
\begin{equation}\label{eq:q-root-depth}
        q=A+s+C.
\end{equation}
For every $\alpha\in\mathcal B$, the image of $\rho_{2A}^*$ contains
an element lying in the $q$-th congruence subgroup whose leading term
at level $q$ is $X_\alpha$.  Moreover
\begin{equation}\label{eq:q-below-window}
        q+1\le2A-R.
\end{equation}
\end{proposition}

\begin{proof}
Fix $\alpha\in\mathcal B$. By definition of $s_\alpha$ choose
$h_\alpha\in H$ with
\[
        z_\alpha(h_\alpha)=p^{s_\alpha}u_\alpha,
        \qquad u_\alpha\in\Zp^\times.
\]
Since $\chi(h_\alpha)=1$, one has
$\rho_\infty(h_\alpha)=1$, so $H\subset K_A$. Formula
\eqref{eq:rho-2A} then becomes
\[
        \rho_{2A}^*(h_\alpha)
        =\kappa_A(F(h_\alpha)\bmod p^A).
\]
Thus $F(h_\alpha)\bmod p^A\in M_A$.  The vector
$F(h_\alpha)$ is supported on the simple-root lines and its
$X_\alpha$-coefficient has valuation exactly $s_\alpha$.  Since $M_A$
is Galois-stable, Lemma~\ref{lem:projector} applies and gives
\[
        p^{s_\alpha+C}X_\alpha\in M_A.
\]
Multiplying by $p^{s-s_\alpha}$ gives
$p^{s+C}X_\alpha\in M_A$.  Under the identification of the
square-zero kernel, Lemma~\ref{lem:kappa-depth} says that the
finite-level image contains an element whose first nonzero congruence
term occurs at depth $A+s+C=q$ and equals $X_\alpha$.

Finally, $A>s+C+R+2$ implies
\[
 q+1=A+s+C+1\le2A-R.
\]
The root data therefore lie strictly below the lower truncation at
which the relative lifting tower will begin.
\end{proof}

\section{Proof of the main theorem}\label{sec:main-proof}

We now lift the finite-level representation of the preceding
section to characteristic zero.

Recall that $R$ was chosen so that
\[
 H^2(G_{\{p\}},T_R)\longrightarrow H^2(G_{\{p\}},T_{R-1})
\]
is injective. Set
\begin{equation}\label{eq:j0}
        j_0=2A-R.
\end{equation}
The choice of $A$ gives $j_0>A>R$.  Moreover
$\rho_{2A}^*$ reduces modulo $p^A$ to $\rho_A$, and hence modulo $p^R$
to $\rho_R$.

\begin{proposition}\label{prop:tower}
For every $j\ge j_0$ there exist continuous representations
\[
 \tau_j:G_{\{p\}}\longrightarrow\GG(\Z/p^j\Z)
\]
and
\[
 \varrho_{j+R}:G_{\{p\}}
 \longrightarrow\GG(\Z/p^{j+R}\Z)
\]
such that
\begin{enumerate}[label=\textup{(\arabic*)}]
\item $\varrho_{j+R}\pmod{p^j}=\tau_j$;
\item $\tau_{j+1}\pmod{p^j}=\tau_j$;
\item $\varrho_{j+R}\pmod{p^R}=\rho_R$;
\item $\tau_{j_0}=\rho_{2A}^*\pmod{p^{j_0}}$.
\end{enumerate}
\end{proposition}

\begin{proof}
At the initial stage put
$\varrho_{j_0+R}=\rho_{2A}^*$ and define $\tau_{j_0}$ by reduction.
Since $j_0+R=2A$, properties (1), (3), and (4) hold.

Suppose now that $\tau_j$ and $\varrho_{j+R}$ have been constructed.
Define $\tau_{j+1}$ to be the reduction of $\varrho_{j+R}$ modulo
$p^{j+1}$. Then $\tau_{j+1}\pmod{p^j}=\tau_j$ and
$\tau_{j+1}\pmod{p^R}=\rho_R$.  Proposition~\ref{prop:relative-lifting}
applies with the fixed width $R$ and gives a representation
\[
 \varrho_{j+R+1}:G_{\{p\}}
 \longrightarrow\GG(\Z/p^{j+R+1}\Z)
\]
whose reduction modulo $p^{j+1}$ is $\tau_{j+1}$.  Reducing further
modulo $p^R$ gives $\rho_R$.  This completes the induction.
\end{proof}

The high-level representations $\varrho_{j+R}$ are auxiliary and need
not be compatible.  The compatibility is only in the lower
truncations $\tau_j$.

\begin{corollary}\label{cor:limit}
There exists a continuous representation
\[
        \rho:G_{\{p\}}\longrightarrow\GG(\Zp)
\]
whose reduction modulo $p^j$ is $\tau_j$ for every $j\ge j_0$.
\end{corollary}

\begin{proof}
The compatible system in Proposition~\ref{prop:tower}(2) defines a
homomorphism to
$\varprojlim_j\GG(\Z/p^j\Z)=\GG(\Zp)$.  Continuity follows from the
continuity of the finite quotients.
\end{proof}

By Proposition~\ref{prop:initial-roots}, for every
$\alpha\in\mathcal B$ there is $g_\alpha\in G_{\{p\}}$ such that the
finite-level element $\rho_{2A}^*(g_\alpha)$ has leading term
$X_\alpha$ at depth $q$, and $q+1\le j_0$.

\begin{corollary}\label{cor:persistent-roots}
For every $\alpha\in\mathcal B$ there exists
$g_\alpha\in G_{\{p\}}$ such that
\[
        \rho(g_\alpha)\in U_q(\GG)
        \quad\text{and}\quad
        \ell_q(\rho(g_\alpha))=X_\alpha.
\]
\end{corollary}

\begin{proof}
The final representation $\rho$ and $\rho_{2A}^*$ have the same
reduction modulo $p^{j_0}$.  Since $q+1\le j_0$, they have the same
reduction through the level at which the $q$-th leading term is
measured.  Hence each simple-root leading term constructed in
Proposition~\ref{prop:initial-roots} survives unchanged in $\rho$.
\end{proof}

Put $\mathcal H=\rho(G_{\{p\}})$ and $\mathcal D=\overline{[\mathcal H,\mathcal H]}$. Every commutator in $\GG(\Zp)$ lies in $\Gder(\Zp)$, so
$\mathcal D\subset\Gder(\Zp)$.

Let $\bar g$ generate $\Fp^\times$. Since the mod-$p$ cyclotomic
character is surjective, choose $\sigma_0\in G_{\{p\}}$ with
$\chibar(\sigma_0)=\bar g$ and put $t=\rho(\sigma_0)$.  The reduction
of $\rho$ modulo $p$ agrees with that of the toral representation
$\rho_\infty$, because $\rho\pmod{p^R}=\rho_R$.  Thus conjugation by
$t$ on the $\alpha$-root line modulo $p$ is multiplication by
$\bar g^{d_\alpha}$.

\begin{lemma}\label{lem:simple-roots-derived}
For every $\alpha\in\mathcal B$, $X_\alpha\in\Phi_q(\mathcal D)$.
\end{lemma}

\begin{proof}
Choose $x_\alpha=\rho(g_\alpha)$ as in
Corollary~\ref{cor:persistent-roots}.  The commutator
$[t,x_\alpha]$ belongs to $\mathcal D$ and lies in $U_q(\Gder)$.  Its
leading term is
\[
        (\bar g^{d_\alpha}-1)X_\alpha.
\]
For $\alpha\in\mathcal B$, the integer $d_\alpha$ is odd.  Since the
order $p-1$ of $\bar g$ is even, $p-1$ cannot divide $d_\alpha$.
Hence $\bar g^{d_\alpha}\neq1$, so the displayed scalar is nonzero in
$\Fp$.  Because $\Phi_q(\mathcal D)$ is an $\Fp$-subspace, division
by this nonzero scalar yields $X_\alpha\in\Phi_q(\mathcal D)$.
\end{proof}

This argument is worth noting: it uses only the parity of the
simple-root weights and therefore imposes no lower bound on $p$ beyond
oddness.

Let $h_{\max}$ be the maximum height of a positive root.  Starting from
$X_{\pm\alpha}$ for the simple roots, the Chevalley relations generate
every root vector.  We include the depth bookkeeping because it is
needed to place all directions in one common congruence layer.

\begin{lemma}\label{lem:all-root-vectors}
Let $\beta\in\Phi$. Then
\[
        X_\beta\in
        \Phi_{|\Ht(\beta)|q}(\mathcal D).
\]
\end{lemma}

\begin{proof}
We argue first for positive roots by induction on $h=\Ht(\beta)$.  If
$h=1$, the root is simple and the assertion is
Lemma~\ref{lem:simple-roots-derived}.  Suppose $h>1$.  There is a
simple root $\alpha\in\Delta$ such that
$\gamma=\beta-\alpha$ is a positive root of height $h-1$.  By
induction,
\[
        X_\alpha\in\Phi_q(\mathcal D),
        \quad\text{and}\quad
        X_\gamma\in\Phi_{(h-1)q}(\mathcal D).
\]
The commutator property gives
\[
 [X_\alpha,X_\gamma]
 \in\Phi_{hq}(\mathcal D).
\]
By the Chevalley relation,
\[
 [X_\alpha,X_\gamma]
 =N_{\alpha,\gamma}X_\beta.
\]
The Lie-good hypothesis says that
$N_{\alpha,\gamma}$ is nonzero modulo $p$.  Since
$\Phi_{hq}(\mathcal D)$ is an $\Fp$-space, division by this scalar
gives $X_\beta$ at the required level.  The negative roots are handled
in exactly the same way using the negative simple roots.
\end{proof}

For every simple root $\alpha$ we also have
\[
 [X_\alpha,X_{-\alpha}]=H_\alpha.
\]
Hence
\begin{equation}\label{eq:coroot-depth}
        H_\alpha\in\Phi_{2q}(\mathcal D)
        \quad \text{for }\quad \alpha\in\Delta.
\end{equation}
By the Lie-good hypothesis, the reductions of all root vectors and
these simple coroot directions span
$\frakg^{\mathrm{der}}_{\Fp}$.

Put
\begin{equation}\label{eq:r0}
        r_0=\max\{2,h_{\max}\}q.
\end{equation}
Repeated use of the $p$-power inclusion
$\Phi_r(\mathcal D)\subset\Phi_{r+1}(\mathcal D)$ moves every root
vector and every $H_\alpha$ from the level at which it was produced to
the common level $r_0$.  We have proved:

\begin{proposition}\label{prop:full-lie-layer}
One has
\[
        \Phi_{r_0}(\mathcal D)
        =\frakg^{\mathrm{der}}_{\Fp}.
\]
\end{proposition}

\begin{lemma}\label{lem:successive-lifting}
Let $D\subset\Gder(\Zp)$ be a closed subgroup. Suppose that for some
$r\ge1$ the natural map
\[
        D\cap U_r(\Gder)
        \longrightarrow
        U_r(\Gder)/U_{r+1}(\Gder)
        \simeq\frakg^{\mathrm{der}}_{\Fp}
\]
is surjective. Then
\[
        U_r(\Gder)\subset D.
\]
\end{lemma}

\begin{proof}
First, surjectivity persists at every deeper level.  Suppose it holds
at level $s\ge r$ and take a prescribed class
$X\in\frakg^{\mathrm{der}}_{\Fp}$ at level $s+1$. Choose
$x\in D\cap U_s(\Gder)$ with leading term $X$.  By
\eqref{eq:p-power-leading-general}, $x^p$ belongs to
$D\cap U_{s+1}(\Gder)$ and has the same leading term $X$ at level
$s+1$.  Induction gives surjectivity at every level.

Now fix $u\in U_r(\Gder)$.  Construct elements $P_s\in D$ such that
\[
        uP_s^{-1}\in U_s(\Gder).
\]
Start with $P_r=1$.  If $P_s$ has been constructed, put
$E_s=uP_s^{-1}$.  Choose $d_s\in D\cap U_s(\Gder)$ with the same
class as $E_s$ in $U_s/U_{s+1}$.  Then
$E_sd_s^{-1}\in U_{s+1}$.  Set
\[
        P_{s+1}=d_sP_s.
\]
Since $P_{s+1}^{-1}=P_s^{-1}d_s^{-1}$,
\[
 uP_{s+1}^{-1}=E_sd_s^{-1}\in U_{s+1}.
\]
Thus $P_s\to u$.  Closedness of $D$ gives $u\in D$.  Since $u$ was
arbitrary, $U_r(\Gder)\subset D$.
\end{proof}

\begin{proposition}\label{prop:large-image}
The representation $\rho$ constructed above satisfies
\[
        U_{r_0}(\Gder)
        \subset\mathcal D
        \subset\rho(G_{\{p\}}).
\]
\end{proposition}

\begin{proof}
Proposition~\ref{prop:full-lie-layer} says exactly that
$\mathcal D\cap U_{r_0}(\Gder)$ surjects onto the first congruence
quotient at depth $r_0$.  Lemma~\ref{lem:successive-lifting} therefore
gives $U_{r_0}(\Gder)\subset\mathcal D$.  The second inclusion is
immediate from the definition of the closed commutator subgroup.
\end{proof}

We first record two elementary topological facts used in passing from
the derived group to the full reductive group.

\begin{lemma}\label{lem:one-dimensional-torus-open}
Let $p$ be odd and let $C$ be a closed subgroup of the one-dimensional
split torus $\Gm(\Zp)=\Zp^\times$.  If the projection of $C$ to the
pro-$p$ factor $1+p\Zp$ is nontrivial, then $C$ is open in
$\Zp^\times$.
\end{lemma}

\begin{proof}
For odd $p$ one has
\[
        \Zp^\times=\mu_{p-1}\times(1+p\Zp).
\]
The $p$-adic logarithm identifies $1+p\Zp$ topologically with the
additive group $p\Zp$, and every nonzero closed subgroup of $p\Zp$
is of the form $p^a\Zp$ for some $a\ge1$.  Thus the pro-$p$
projection of $C$ is open in $1+p\Zp$.  Since the complementary
factor $\mu_{p-1}$ is finite, $C$ has finite index in
$\Zp^\times$ and is therefore open.
\end{proof}

\begin{lemma}\label{lem:derived-and-quotient-open}
Let $H\subset\GG(\Zp)$ be a closed subgroup.  Suppose that
$H\cap\Gder(\Zp)$ is open in $\Gder(\Zp)$ and that the image of
$H$ in $(\GG/\Gder)(\Zp)$ is open.  Then $H$ is open in
$\GG(\Zp)$.
\end{lemma}

\begin{proof}
Let $q:\GG(\Zp)\to(\GG/\Gder)(\Zp)$ be the quotient map and put
$V=q(H)$, which is open by hypothesis.  Then $q^{-1}(V)$ is open in
$\GG(\Zp)$.  If $g\in q^{-1}(V)$, choose $h\in H$ with
$q(h)=q(g)$.  Then $h^{-1}g\in\Gder(\Zp)$, so
\[
        q^{-1}(V)=H\,\Gder(\Zp).
\]
Consequently
\[
 [q^{-1}(V):H]
 \le [\Gder(\Zp):H\cap\Gder(\Zp)]<\infty.
\]
Thus $H$ has finite index in the open subgroup $q^{-1}(V)$, and hence
$H$ is open in $\GG(\Zp)$.
\end{proof}
We now prove the main theorem.
\begin{proof}[Proof of Theorem \ref{thm:main}]If $\GG$ is semisimple, Proposition~\ref{prop:large-image} already
shows that $\rho(G_{\{p\}})$ is open in $\GG(\Zp)$.  Suppose now that
$\dim Z(\GG)=1$ and put $\mathbf S=\GG/\Gder$.  The final
representation $\rho$ agrees with the initial finite-level
representation modulo $p^{j_0}$, and the latter agrees with
$\rho_\infty$ modulo $p^A$.  Hence the image of $\rho$ in
$\mathbf S(\Zp)$ has the same reduction modulo $p^A$ as the nonzero
cyclotomic character
$\lambda^{\mathrm{ab}}\circ\chi$.  By the additional choice of $A$ in
Section~\ref{sec:initial}, this reduction already has nontrivial
pro-$p$ image. A closed subgroup of the one-dimensional $p$-adic torus
with nontrivial pro-$p$ image is open by
Lemma~\ref{lem:one-dimensional-torus-open}.  Thus the image of
$\rho$ in the central quotient is open.  Together with the principal
congruence subgroup in $\Gder(\Zp)$, Lemma~\ref{lem:derived-and-quotient-open}
proves that $\rho(G_{\{p\}})$ is open in $\GG(\Zp)$.

It remains to obtain infinitely many representations.  The construction
works for every sufficiently large choice of the width parameter $R$:
if $R$ is enlarged, the cocharacter $\lambda$, the finite group
$H^2(G_{\{p\}},T)$, the cocycles $z_\alpha$, and the constants $s$ and
$C$ remain fixed; one simply chooses $A$ still larger so that
\eqref{eq:A-choice} holds.  The resulting characteristic-zero
representation satisfies
\[
        \rho^{(R)}\equiv\rho_\infty\pmod{p^R}.
\]
Every open-image representation $\rho^{(R)}$ differs from the toral
representation $\rho_\infty$, because the latter has image contained in
$\TT(\Zp)$.  Suppose finitely many distinct representations
$\rho^{(R_1)},\ldots,\rho^{(R_m)}$ have already been chosen.  For each
$i$, choose a level $b_i$ at which
$\rho^{(R_i)}\not\equiv\rho_\infty\pmod{p^{b_i}}$.  Choose a new width
$R>\max_i b_i$ and carry out the construction.  Then the new
representation is congruent to $\rho_\infty$ modulo every $p^{b_i}$,
so it cannot equal any of the previous representations.  Induction
produces infinitely many pairwise distinct representations.

In fact, the same construction gives infinitely many
$\GG(\Zp)$-conjugacy classes, which is the stronger and more natural
form of the infinitude statement.  Suppose that
$\rho_1,\ldots,\rho_m$ have already been chosen in pairwise distinct
$\GG(\Zp)$-conjugacy classes.  For each $i$ consider, for $b\ge1$,
\[
 \mathcal C_{i,b}
 =\left\{g\in\GG(\Zp):
    g\rho_i g^{-1}\equiv\rho_\infty\pmod{p^b}
  \right\}.
\]
Each $\mathcal C_{i,b}$ is closed in the compact group
$\GG(\Zp)$, and the sets are nested as $b$ increases.  If
$\mathcal C_{i,b}$ were nonempty for every $b$, compactness would give
an element in their intersection.  Such an element would conjugate
$\rho_i$ to $\rho_\infty$ exactly.  This is impossible, because
$\rho_i$ has open derived image whereas $\rho_\infty$ has image in
the torus $\TT(\Zp)$.  Hence for each $i$ there is an integer $b_i$
such that $\mathcal C_{i,b_i}=\varnothing$.

Choose the next width $R$ larger than every $b_i$.  The newly
constructed representation $\rho^{(R)}$ satisfies
$\rho^{(R)}\equiv\rho_\infty\pmod{p^R}$, and therefore cannot be
conjugate to any $\rho_i$: otherwise a conjugating element would
belong to $\mathcal C_{i,b_i}$.  Induction produces infinitely many
pairwise nonconjugate representations.
\end{proof}

We finish by checking every hypothesis in the general theorem for
$\GG=\GL_n$. 

\begin{proof}[Proof of Corollary\ref{cor:GLn}]

Take the diagonal torus.  Its cocharacter lattice is $\Z^n$, and for
\[
        \lambda(t)=\operatorname{diag}
        (t^{a_1},\ldots,t^{a_n})
\]
the root $e_i-e_j$ has weight $a_i-a_j$.  Let $\cE_p$ be the finite
exceptional set from Proposition~\ref{prop:finite-exceptional}.  We
choose $a_1,\ldots,a_n$ recursively so that
\begin{enumerate}[label=\textup{(\arabic*)}]
\item $a_i\equiv i\pmod2$;
\item $a_i-a_j\notin\cE_p\cup\{0\}$ for $i\neq j$;
\item all ordered differences $a_i-a_j$, $i\neq j$, are pairwise
      distinct;
\item $a_1+\cdots+a_n\neq0$.
\end{enumerate}
At the $(r+1)$-st stage the prescribed parity leaves an infinite
arithmetic progression of possible values for $a_{r+1}$.  Avoiding a
root weight in the finite exceptional set excludes finitely many
values.  Equality of a new difference with an old difference gives a
linear equation in $a_{r+1}$ and therefore excludes one value.  Equality
between two new ordered differences gives either an equality already
forced by their indices or an equation of the form
$2a_{r+1}=a_i+a_j$, again excluding at most one value.  Thus only
finitely many integers of the prescribed parity are forbidden.  At the
last stage we also avoid the single value which would make the total
sum zero.  This proves the required choice for every $n$ and every odd
$p$.

The adjacent differences
\[
        d_i=a_i-a_{i+1}
\]
are odd.  Hence the positive and negative simple roots satisfy the odd
cocharacter condition.  All root weights avoid $\cE_p$, so the
rational obstruction vanishing used in the general proof holds.  The
nonzero sum $a_1+\cdots+a_n$ says that the determinant of the toral
representation is the nonzero cyclotomic power
\[
        \det\rho_\infty=\chi^{a_1+\cdots+a_n},
\]
which has open image in $\Zp^\times$.

It remains only to check the Lie-good condition.  In the standard
matrix realization take
\[
        X_{e_i-e_j}=e_{ij}.
\]
Whenever $i,j,k$ are distinct,
\[
        [e_{ij},e_{jk}]=e_{ik},
\]
so every nonzero root-string structure constant needed to generate the
higher roots is $\pm1$.  These constants remain invertible modulo every
prime.  Opposite roots satisfy
\[
        [e_{ij},e_{ji}]=e_{ii}-e_{jj}.
\]
The off-diagonal matrix units together with
\[
        e_{11}-e_{22},\ldots,e_{n-1,n-1}-e_{nn}
\]
span $\mathfrak{sl}_n(\Fp)$ for every prime $p$.  This remains true
when $p\mid n$: the trace-zero diagonal space still has dimension
$n-1$ and is spanned by the adjacent differences.  Thus no prime is
excluded by the Chevalley or coroot calculations.

Finally, the $p$-power congruence calculation used in
\eqref{eq:p-power-leading-general} requires only that $p$ be odd.  No
step of the proof compares $p$ with $n$, and no class-group or
irregularity condition remains.  Theorem~\ref{thm:main} therefore
applies to $\GL_n$ for every odd prime $p$ and every $n>1$, proving
Corollary~\ref{cor:GLn}.
\end{proof}

\bibliographystyle{amsplain}
\bibliography{references}

\end{document}